\documentclass[12pt]{article}
\title{Interval $k$-Graphs and Orders}

\author{David E. Brown\thanks{david.e.brown@usu.edu.}\\
Department of Mathematics and Statistics\\
    Utah State University\\
    Logan, UT 84341-3900\\
    U.S.A.
\and
 Larry J. Langley\thanks{llangley@pacific.edu}\\
 Mathematics Department \\
 University of the Pacific\\
 Stockton, CA 95211\\
 U.S.A.
\and
 Breeann M. Flesch\thanks{fleschb@wou.edu}\\
 Mathematics Department\\
 Western Oregon University\\
 Monmouth, OR 97361\\
 U.S.A.
}

\usepackage{graphicx} 
\usepackage{amsmath}
\usepackage{amssymb}
\usepackage{theorem}
\usepackage{multicol}
\usepackage{psfrag}
\newtheorem{corollary}{Corollary}[section]

\newtheorem{theorem}{Theorem}[section]

\newenvironment{proof}{\noindent{\bf Proof.}}{\nopagebreak[4]\hfill \rule{2mm}{2mm}}
\newtheorem{proposition}{Proposition}[section]

\begin{document}

\maketitle
\date

\begin{abstract} An interval $k$-graph is the intersection graph of a
family $\mathcal{I}$ of intervals of the real line partitioned into at most $k$
classes with vertices
adjacent if and only if their corresponding intervals intersect and belong to
different classes.
In this paper we discuss the interval $k$-graphs that are the incomparability graphs
of orders; i.e., cocomparability interval $k$-graphs or interval $k$-orders.
Interval $2$-orders have been characterized in many ways, but we show that analogous characterizations
do not carry over to interval $k$-orders, for $k > 2$.
We describe the structure of interval $k$-orders, for any $k$, characterize the interval $3$-orders
(cocomparability interval $3$-graphs) via one forbidden suborder (subgraph), and state a conjecture
for interval $k$-orders (any $k$) that would characterize them via two forbidden suborders.
\end{abstract}

\section{Introduction}

We discuss finite simple graphs and use the notation $G=(V,E)$ to denote a graph with vertex set $V = V(G)$ and edge set $E = E(G)$.
For the complement of graph $G$ we use the notation $\overline{G}$.

An \emph{ordered set} (or \emph{strict partial order}) is a pair $P = (X, \prec)$ consisting of a ground set $X$ and a binary relation $\prec$ on $X$
that is irreflexive, transitive and therefore antisymmetric.
If neither $x \prec y$ nor $y \prec x$ occurs in $P$, we say $x$ and $y$ are \emph{incomparable} and write $x \| y$, otherwise they are \emph{comparable}.
Two graphs are naturally associated to the order $P$: its comparability graph and its incomparability graph.
The graph $G = (V,E)$ is the \emph{comparability graph} of $P$ if $V= X$ and, for $x,y \in X$, $xy \in E$ if and only if $x$ and $y$ are comparable.
The \emph{incomparability graph} of $P$ has vertex set $X$ and vertices $x$ and $y$ adjacent if and only if $x \| y$ in $P$.
Note that although $x\|x$ for any $x \in X$, we choose not to clutter the incomparability graphs with
loops so $x \| x$ does not yield an edge, and of course the complement of $G$ is the incomparability graph of $P$ if
$G$ is the comparability graph of $P$.
When the edges of $G$ can be given a transitive orientation, $G$ is
the comparability graph of some order and when the edges of $\overline{G}$ can be transitively oriented, $G$ is
called a \emph{cocomparability graph} and is hence the incomparability graph of some order.
Most of the graphs we discuss in this paper are cocomparability graphs.

A family of sets $\mathcal{F} = \{S_1, S_2, \dots, S_n\}$ is an \emph{intersection representation} of a graph
$G$ if $V(G)$ can be put into one-to-one correspondence with $\mathcal{F}$ so that $S_i\cap S_j \neq \O$ if and only if
vertices $u$ and $v$ are adjacent in $G$.
For example, if graph $G$ is from the well-studied class of \emph{interval graphs}, then $G$ is a graph which can be represented
so that $\mathcal{F}$ is a family of intervals of the real line.
In this paper we investigate graphs which can be represented as intersection graphs of intervals of the real line, but unlike interval graphs,
with the property that certain subsets of intervals' intersection information does not correspond to adjacency in the graph.
Specifically, we investigate graphs $G=(V,E)$ for which there is a one-to-one correspondence between $V$ and a collection of
intervals of the real line $\mathcal{I}$ partitioned into what we will call \emph{interval classes} or simply \emph{classes}
so that vertices are adjacent if and only if their corresponding intervals intersect and belong to different classes.
We use $I_v$ to denote the interval corresponding to vertex $v$.
If for $G = (V,E)$ there is such a representation $\mathcal{I}$ partitioned into at most $k$ classes
$\mathcal{I} = \mathcal{I}_1 \cup \mathcal{I}_2 \cup \dots \cup \mathcal{I}_k$, with vertices $u$ and $v$ adjacent in $G$ if and only if
$I_u \cap I_v \neq \O$ and $I_u, I_v$ belong to different interval classes, then $G$ is an \emph{interval $k$-graph}.
The collection $\mathcal{I}$ with the partition into classes will be called an \emph{interval $k$-representation} or simply a \emph{representation}
if the context precludes ambiguity.
Note that the set of vertices corresponding to intervals from any class induce an independent set.
In the case $k=2$, this class has been called the \emph{interval bigraphs} and has enjoyed considerable attention recently, see
for example \cite{Bro,DasRoySenWes, HelHua} and their references for more.
Our focus is on the interval $k$-graphs that are cocomparability graphs and hence those interval
$k$-graphs which give rise to a strict partial order we will call an \emph{interval $k$-order}.

The class of probe interval graphs is another class of intersection graphs that has enjoyed recent attention.
See for example \cite{BroLunAus, BroLunShe, McMWanZha, McKMcM, She}.
A probe interval graph is another interval-intersection graph in which certain intervals' intersection information is ignored.
A graph $G$ is a \emph{probe interval graph} if its vertices can be partitioned into sets $P$ (\emph{probes}) and $N$ (\emph{nonprobes}) with an
interval of the real line corresponding to each vertex, and vertices adjacent if and only if their corresponding intervals
intersect and at least one is a probe.
In \cite{BroLanPIO} the probe interval graphs that are cocomparability graphs, and hence
incomparability graphs of \emph{probe interval orders}, are characterized in various ways.
One characterization states the collection of intervals corresponding to the nonprobes has the property that no interval contains another properly
while the probes' intervals are not restricted.
If a probe interval graph has such a representation, it is called a \emph{nonprobe-proper probe interval graph}.

\begin{theorem}\label{equivalence} \emph{(Brown, Langley, \cite{BroLanPIO})} The graph
$G$ is a cocomparability graph whose vertices can be partitioned into sets $P$ and $N$ with $N$ an independent set, and
every 4-cycle alternates between $N$ and $P$ if and only if $G$ is a nonprobe-proper probe interval graph.
\end{theorem}

\noindent The results we develop here are similar to those in \cite{BroLanPIO} in that we are (1) trying to do for interval $k$-orders what
Fishburn did for interval orders (see \cite{Fis}) and (2) we show that the mechanism by which an interval $k$-graph
may contain an obstruction to being a cocomparability graph is having an interval from some interval class contain
another from that class properly.
Whence we define the following restricted class of interval $k$-graphs.
Let $G$ be an interval $k$-graph with an interval representation $\mathcal{I}$
partitioned into classes $\mathcal{I}_1, \mathcal{I}_2, \dots, \mathcal{I}_k$ so that no interval from any class contains another from its class properly.
We call such an interval $k$-graph a \emph{class-proper interval $k$-graph}, and the collection of intervals representing it a
\emph{class-proper representation}.

Thanks to the monumental characterization of transitively orientable graphs by Gallai (Theorem \ref{Gallai} below) we can
find interval $k$-graphs which are not cocomparability graphs by identifying odd asteroids.
An \emph{odd asteroid} is a sequence $v_0, P_0, v_1, P_1$, $v_2, \dots, v_{2n}, P_{2n},v_0$, where $v_0, v_1, \dots, v_{2n}$
are distinct vertices, $P_i$ is a $v_i,v_{i+1}$-path, and $N(v_i)\cap P_{i + n} = \O$, where subscripts are taken modulo $2n+1$.
If a graph has a set of $2n+1$ vertices on which an odd asteroid exists, we will call the set of vertices a $(2n+1)$-asteroid.
The graph, which we refer to as $T_2$ in Figure \ref{fig:T2} has a 3-asteroid (also known as an \emph{asteroidal triple}) on the vertices $a, b,$ and $c$.
The graph $G$ of Figure \ref{fig:not_coco_no_at} has a 5-asteroid on $v_0 = q, v_1 = x, v_2 = a, v_3 = c,$ and $v_4 = b$, whence it is not
a cocomparability graph, and neither is $T_2$, via the theorem of Gallai we mentioned.

\begin{theorem}\label{Gallai} \emph{(Gallai, \cite{Gal})} The complement of a graph $G$ has a transitive orientation if and only
if $G$ has no odd asteroid. \end{theorem}

\noindent But $G$ and $T_2$ have interval $k$-representations as illustrated and so
\emph{the class of interval $k$-graphs is not contained in the class of cocomparability graphs.}
The converse containment relationship also does not hold; this will be shown below.

\begin{figure}[h]
\begin{center}
\psfrag{a}{{\small$a$}}
\psfrag{b}{{\small$b$}}
\psfrag{c}{{\small$c$}}
\psfrag{d}{{\small$d$}}
\psfrag{e}{{\small$e$}}
\psfrag{x}{{\small$x$}}
\psfrag{y}{{\small$y$}}
\psfrag{T2}{$T_2$}
\includegraphics[scale=.4]{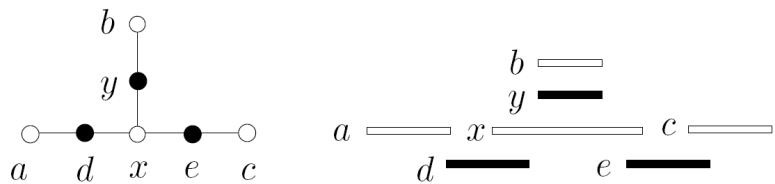}
\caption{An interval 2-graph with an asteroidal triple and
interval representation illustrating how an asteroidal triple requires
some interval to contain another from the same class properly.}\label{fig:T2}
\end{center}
\end{figure}

\begin{figure}[h]
\begin{center}
\psfrag{a}{{\small$a$}}
\psfrag{b}{{\small$b$}}
\psfrag{c}{{\small$c$}}
\psfrag{p}{{\small$p$}}
\psfrag{q}{{\small$q$}}
\psfrag{x}{{\small$x$}}
\psfrag{y}{{\small$y$}}
\psfrag{I1}{{\small$\mathcal{I}_1$}}
\psfrag{I2}{{\small$\mathcal{I}_2$}}
\psfrag{I3}{{\small$\mathcal{I}_3$}}
\psfrag{G}{$G$}
\psfrag{Gc}{$\overline{G}$}
\includegraphics[scale=.6]{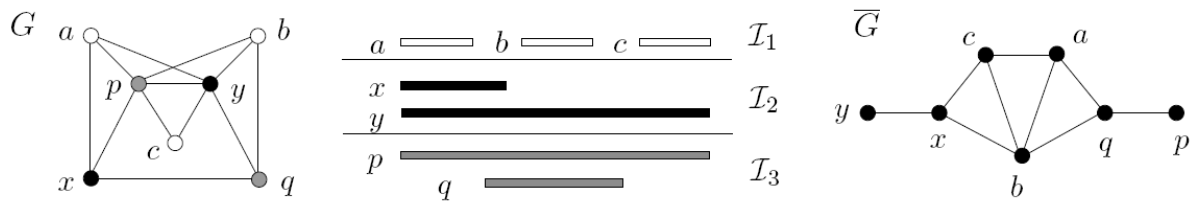}
\caption{An interval 3-graph which is not a cocomparability graph
and its representation showing it is not class-proper.}\label{fig:not_coco_no_at}
\end{center}
\end{figure}

Any probe interval graph is an interval $k$-graph, as was shown in \cite{Bro}, but in Figure \ref{fig:cocoIkG_notPIG} we
have a cocomparability interval $3$-graph $M$ which is not a probe interval graph.  The poset $P$ corresponds to a
transitive orientation of the complement of $M$.
Therefore \emph{the class of cocomparability interval $k$-graphs contains the
class of cocomparability probe interval graphs}.

\begin{figure}[htbp]
\begin{center}
\psfrag{a}{{\small$a$}}
\psfrag{b}{{\small$b$}}
\psfrag{p}{{\small$p$}}
\psfrag{q}{{\small$q$}}
\psfrag{x}{{\small$x$}}
\psfrag{y}{{\small$y$}}
\psfrag{I1}{{\small$\mathcal{I}_1$}}
\psfrag{I2}{{\small$\mathcal{I}_2$}}
\psfrag{I3}{{\small$\mathcal{I}_3$}}
\psfrag{M}{$M$}
\psfrag{P}{$P$}
\includegraphics[scale=.6]{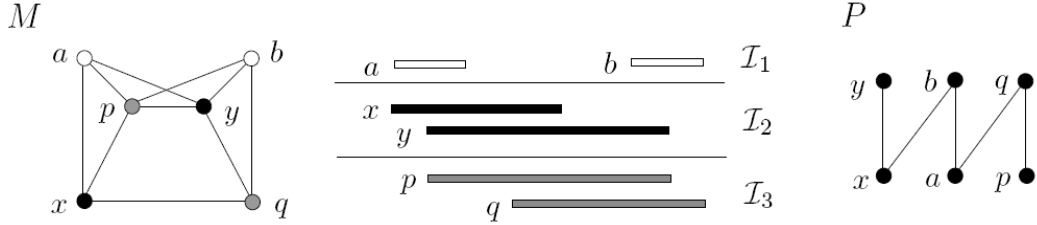}
\caption{A cocomparability interval 3-graph $M$ which is not a probe interval graph, its class-proper interval representation,
and the Hasse diagram of the strict order $P$ obtained from a transitive orientation of $\overline{M}$ by
the methods in Theorem \ref{transCPIkG}.}\label{fig:cocoIkG_notPIG}
\end{center}
\end{figure}

A beautiful characterization of cocomparability graphs by Golumbic, Rotem, and Urrutia (Theorem \ref{func=coco} below) shows they are precisely the
function graphs.
A \emph{function graph} $G$ is the intersection graph of a family of curves of continuous functions $f_i : [a,b] \to \mathbb{R}$;
that is, vertices $i$ and $j$ are adjacent if and only if $f_i(x) = f_j(x)$ for some $x \in [a,b]$.  It is easy to see that every
function graph is a cocomparability graph: orient $i \to j$ in $\overline{G}$ if $f_i(y) < f_j(y)$ for all $y \in [a,b]$.

\begin{theorem}\label{func=coco} \emph{(Golumbic, Rotem, Urrutia, \cite{GolRotUrr})} A graph is a function graph
if and only if it is a cocomparability graph.\end{theorem}

\noindent A graph is \emph{weakly chordal} if neither it nor its complement contains an induced cycle on five or more
vertices as an induced subgraph.

\begin{theorem}\label{IkGweakchord}\emph{(Brown, \cite{Bro})} If $G$ is an interval $k$-graph, then $G$ is weakly chordal. \end{theorem}

\noindent In Figure \ref{fig:func_not_IkG} we have a function representation of the complement of a 6-cycle, which is not an
interval $k$-graph, by Theorem \ref{IkGweakchord}; therefore \emph{there is no containment relationship between interval $k$-graphs
and cocomparability graphs.}

\begin{figure}[h!]
\begin{center}
\includegraphics[scale=0.6]{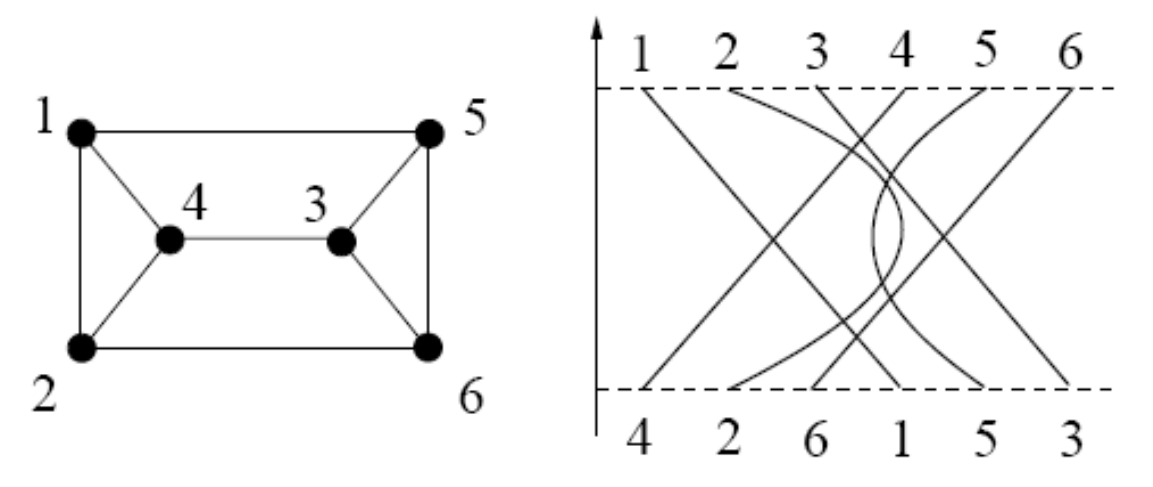}
\caption{A function representation of the complement of a 6-cycle, which is not an interval
$k$-graph by Theorem \ref{IkGweakchord}.}\label{fig:func_not_IkG}
\end{center}
\end{figure}

\section{Interval $2$-Graphs and Orders}

In the next section, we show that if an interval $k$-graph, for $k \geq 3$, has a class-proper representation, then
it is a cocomparability graph.
In this section we show that the well-known and well-studied classes of bipartite cocomparability graphs,
bipartite permutation graphs, and proper and unit interval bigraphs are precisely the class-proper interval $2$-graphs.
We also show that for $k > 2$ none of the characterizations in
Theorem \ref{theorem:charPIBG} extend to proper interval $k$-graphs or class-proper interval $k$-graphs.
Furthermore, for $k > 2$, the classes proper, unit, and class-proper interval $k$-graphs are different.

To prove that class-proper interval $2$-graphs are precisely the cocomparability interval $2$-graphs, we will take
the circuitous route of proving the following cycle of implications, referring to statements in Theorem \ref{theorem:charPIBG}:
\emph{1} $\implies$ \emph{2} $\implies$ \emph{3} $\implies$ \emph{4} $\implies$ \emph{5} $\implies$ \emph{1}.
In so doing we will establish a characterization via the existence of an ordering of the graph's vertices, statement \emph{5},
which is stronger than the strong ordering referred to in statement \emph{6}.
This is the most efficient way we could prove this, possibly due to the fact that the class-proper restriction is not much of one;
to wit, intervals in one interval class could be distinct points while those in the other be unit-length intervals for example.
In the interest of brevity we will direct the reader to the literature where the results and all definitions can be found; we will give
only the essential ones.

A bipartite graph $G$ is a \emph{unit interval bigraph} if it is an interval 2-graph which has a representation where all intervals
have identical length.
$G$ is a \emph{proper interval bigraph} if $G$ is an interval 2-graph which has a representation
in which no interval contains another properly.
A graph $H$ is a \emph{permutation graph} if $V(H) = \{1,2,\dots, n\}$ and there is
a permutation $(\pi_1, \pi_2, \dots, \pi_n)$ of the numbers $V(H)$ such that vertices are adjacent if and only if the numbers are in
reversed order in the permutation.
Equivalently, and this is the definition we will use, a permutation graph $G$ may be defined as the
intersection graph of line segments $\{\ell_v:v\in V(G)\}$ contained
in the space between parallel line segments $L_1$ and $L_2$, we will call \emph{channels}.
For the point where $\ell_v$ intersects $L_i$ we will use $p(v;i)$.
For the intervals in an interval representation, whether it be unit, proper, or class-proper, we will use $I(v) = [L(v), R(v)]$
to denote the interval corresponding to vertex $v$.

\begin{theorem}\label{theorem:charPIBG} Let $G$ be a bipartite graph.  The following are equivalent:\\
(1) $G$ is a unit interval bigraph \emph{\cite{BraSpiSte, BraSpiSte2, HelHua, SanSen}};\\
(2) $G$ is a proper interval bigraph \emph{\cite{HelHua}};\\
(3) $G$ is a class-proper interval 2-graph;\\
(4) $G$ is a permutation graph \emph{\cite{BraSpiSte, HelHua}};\\
(5) $V(G)$ can be ordered $(v_1, v_2, \dots, v_n)$ so that whenever $v_iv_k \in E(G)$ and $i<j<k$, $v_j$ is
adjacent to whichever of $\{v_i,v_k\}$ is not in its partite set;\\
(6) $G$ has a strong ordering \emph{\cite{BraSpiSte}};\\
(7) The bipartite adjacency matrix of $G$ has a monotone consecutive arrangement \emph{\cite{SanSen}};\\
(8) $G$ is the comparability graph of a poset of dimension at most 2 \emph{\cite{DusMil, HelHua}};\\
(9) $G$ is the incomparability graph of a poset of dimension at most 2 \emph{\cite{DusMil, HelHua}};\\
(10) $\overline{G}$ is a proper circular arc graph \emph{\cite{HelHua}};\\
(11) $G$ is contains no asteroidal triple \emph{\cite{HelHua}}.\end{theorem}

\begin{proof} ($1 \implies 2$) A unit representation is a proper representation.

($2 \implies 3$) A proper interval bigraph is clearly a class-proper interval bigraph.

($3 \implies 4$) Place a copy of the class-proper interval representation for $G$ on each of the channels
$L_1$ and $L_2$.  With $\mathcal{I}_X$ and $\mathcal{I}_Y$ the intervals corresponding to $G$'s bipartition, create
$\ell_v$, for $v \in X$, by connecting $l(v)$ on $L_1$ with $r(v)$ on $L_2$, and for $v \in Y$, connect $r(v)$ on $L_1$ with
$l(v)$ on $L_2$. The line segments for vertices from the same partite set will not cross because the interval
representation is class proper, and $xy \in E(G)$, for $x \in X$, $y \in Y$, if and only if $r(x) > l(y)$ and $r(y) > l(x)$
if and only if $\ell_x$ crosses $\ell_y$. Thus, the line segments $\{\ell_v:v\in X \cup Y\}$ between the channels form a
permutation representation for $G$.

($4 \implies 5$) Given a segment representation of a bipartite permutation graph, consider the
vertex ordering $v_1, v_2, \dots, v_n$ given by the order of $p(v;1)$, the endpoints of the segments on channel $L_1$.
That is, $i < j$ if $p(v_i;1)$ is left of $p(v_j;1)$.
Consider $v_i,  v_j,  v_k$ with $i < j < k$ and $v_iv_k \in E(G)$.  The segments $\ell_{v_i}$ and $\ell_{v_k}$
intersect before reaching the other channel.  Since $\ell_{v_j}$ starts between $\ell_{v_i}$ and $\ell_{v_k}$,
it cannot reach the other side without intersecting one of them. It does not intersect the segment
for the vertex in its own partite set, so it intersects the other.  We have the vertex ordering desired.

($5 \implies 1$) Suppose the bipartite graph $G$'s vertices have been ordered $(v_1, v_2, \dots, v_n)$ in accord with the
statement \emph{5} and label the partite sets $X = \{v_{i_1}, v_{i_2}, \dots, v_{i_r}\}$ and $Y = \{v_{j_1}, v_{j_2}, \dots, v_{j_s}\}$
so that the indices respect the ordering (e.g., $v_{i_t}v_{j_{q}} \in E(G)$ implies $v_{i_t}v_{j_p}$ if $i_t < j_q < j_p$).
For convenience, drop the indices on the indices and put $v_{i_t} = x_t$ for $1 \leq t \leq r$, and $v_{j_q} = y_q$ for $1 \leq q \leq s$.
We construct a class-proper representation for $G$ by induction on $n = r+s$ creating the set of intervals
$\{I_{v} = [L(v), L(v)+1]: v \in V(G)\}$ with left end-points distinct and respecting the ordering; i.e., $L(v_1)<L(v_2) < \cdots < L(v_n)$.
If $\min\{r,s\} = 1$, the construction is obvious.  Also, since the intervals of isolated vertices are easy to incorporate we will assume
there are none.
Now, assume a class-proper representation has been created for $G$ induced on $\{x_1, \dots, x_r\}\cup \{y_1, \dots, y_{s-1}\}$,
switching the roles of $X$ and $Y$ if necessary.

{\bf Lemma A:} \emph{If $x_iy_j, x_ty_q \in E(G)$ and $i < t$ and $q < j$, then $x_iy_q, x_ty_j \in E(G)$.}

\emph{Proof of lemma:} This is essentially the observation that the ordering restricted to each partite set is the
strong ordering developed in \cite{BraSpiSte}.  By symmetry we may assume $x_i$ is first among $\{x_i, x_t, y_j, y_q\}$.
The possible orderings of these four elements are $(x_i,y_q,y_j,x_t), (x_i, y_q,x_t,y_j),$ and $(x_i,x_t,y_q,y_j)$.
In each case $x_iy_q, x_ty_j \in E(G)$ in virtue of the properties of the ordering.  This proves the lemma.

Define $q$ and $j$ to be the smallest and largest index, respectively, for which $x_ry_q \in E(G)$ and $x_ry_j \in E(G)$.
Define $i$ and $t$ to be the smallest and largest index, respectively, for which $x_iy_s \in E(G)$ and $x_ty_s \in E(G)$.
If $t < r$, then $q = s$ and $x_{t+1}, \dots, x_{r}$ are isolated (vis-a-vis $x_{t+1}y_{s-1} \in E(G) \implies x_ry_s \in E(G)$ by Lemma A);
hence $t =r$.  Similarly, if $j < s$, then $y_{j+1}, \dots, y_s$ are isolated; so $j =s$.
Note that by the properties of the ordering, $y_s$ is adjacent to each of $x_i, x_{i+1},\dots, x_r$ and $x_r$ is adjacent to each of $y_q, \dots, y_s$.
Then by Lemma A the graph induced on $\{x_i, \dots, x_r\} \cup \{y_q, \dots, y_s\}$ is a biclique and hence the intersection of
all the intervals corresponding to these vertices is not empty.  Furthermore, since all endpoints are distinct, this
intersection is not a point.  Therefore there is a point $p$ in this intersection with $p > L(y_{s-1})$, and
defining $I(y_s) = [p, p+1]$ completes the unit interval construction. \end{proof}\\

Theorem \ref{theorem:charPIBG} could be extended with at least seven more statements (cf. \cite{BroLunUIBG}) and
we would like to see a proof incorporating all (at least) eighteen statements characterizing cocomparability interval $2$-graphs into
a cycle of implications using no extraneous results.
Indeed, we have tried to produce such a proof, but statements \emph{10} and \emph{11} have been prohibitive.
So far statement \emph{11} has only tedious proofs with exhaustive case analysis or an appeal to Theorem \ref{Gallai}.
Statement \emph{10} so far requires an appeal to a result of Spinrad in \cite{Spi}.

\section{Autopsy of Theorem \ref{theorem:charPIBG}'s attempted extension to $k > 2$}

Although the classes of unit interval $k$-graphs and proper interval $k$-graphs are identical, see \cite{BroFle},
the analogues of statements in Theorem \ref{theorem:charPIBG} extend no further for proper interval $k$-graphs with $k>2$.
In this section we show that the statements 3, 4, 5, 8, 9, 10, and 11 of Theorem \ref{theorem:charPIBG}
do not necessarily hold for a proper interval $k$-graph, $k > 2$.

First we show that the vertices of any proper (or unit) interval $k$-graph can be ordered as in statement 5 of
Theorem \ref{theorem:charPIBG}, but the ordering does not characterize proper interval $k$-graphs.

\begin{theorem}\label{DBO} If a $k$-partite graph $G$ is a proper or unit interval $k$-graph, then $V(G)$ can be labeled
$v_1, v_2, \dots, v_n$ so that, for $i < j < k$, if $v_iv_k \in E(G)$, then $v_j$ is adjacent to each of $\{v_i,v_k\}$
in a different partite set than $v_j$.
\end{theorem}

\begin{proof} Suppose $G$ is a proper or unit interval $k$-graph with interval representation $\{I(v)= [\ell(v),r(v)]: v\in V(G)\}$.
Label the vertices $v_1, v_2, \dots, v_n$ so that $\ell(v_i) < \ell(v_j)$ if and only if $i < j$.
Suppose $v_iv_k \in E(G)$, where $i < k$, and consider $v_j$ with $i < j < k$.
Since $\ell(v_i) < \ell(v_j) < \ell(v_k)$ and no interval properly contains another, we
know $r(v_i) < r(v_j) < r(v_k)$ and $r(v_i) \geq \ell(v_k)$ since $I(v_i) \cap I(v_k) \neq \O$.
So $I(v_i) \cap I(v_j) \neq \O$ and $I(v_j) \cap I(v_k) \neq \O$.
Therefore $v_j$ is adjacent to whichever of $\{v_i, v_k\}$ is in a different partite set than $v_j$. \end{proof}\\

\noindent The graph in Figure \ref{cockandballs} is not a unit or proper
interval $k$-graph (straightforward to verify, or see \cite{BroFle} or \cite{Bro}), but is labeled in accord with Theorem \ref{DBO}.

\begin{figure}[ht]
\begin{center}
\includegraphics[scale=0.4]{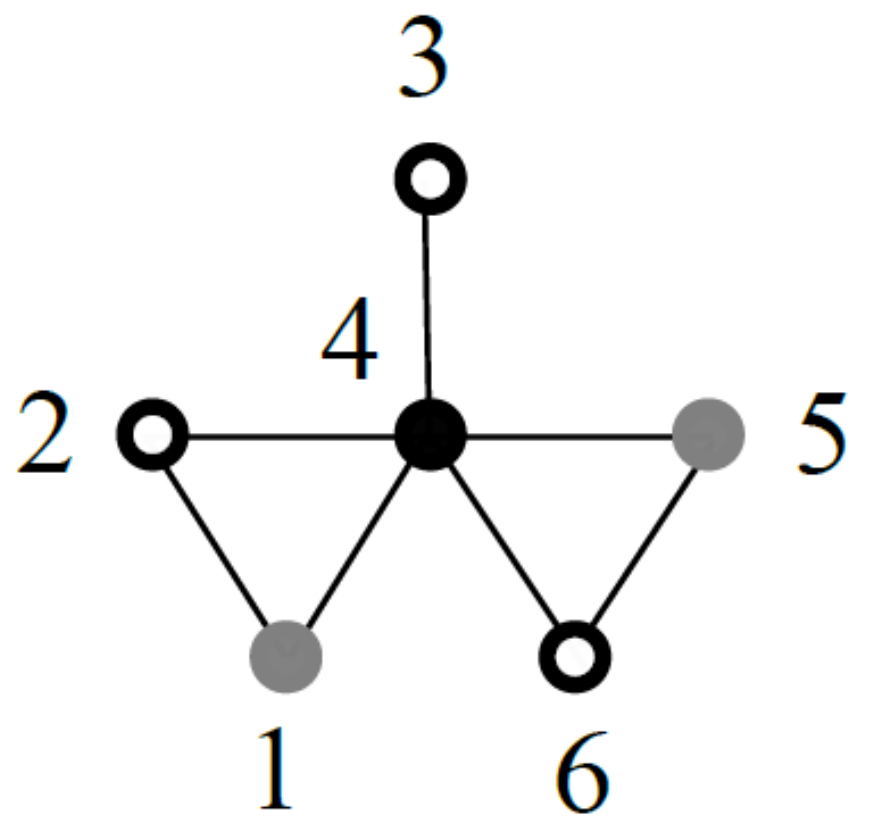}
\caption{An obstruction for a proper interval $k$-graph with vertices labeled in accord with Theorem \ref{DBO}.}\label{cockandballs}
\end{center}
\end{figure}

The next theorem follows from a result of Corneil and others \cite{CorOlaSte} and since proper interval $k$-graphs
are asteroidal triple free \cite{Bro}, but we give a
short proof following from the ordering of Theorem \ref{DBO}.
A \emph{dominating pair} of vertices in a graph $G$ is a pair of vertices that belong to a path $P$ of $G$ such that
every vertex of $G$ belongs to $P$ or is adjacent to a vertex of $P$.

\begin{theorem}\label{dom_pair_in_propIkG} If $G$ is a connected proper or unit interval $k$-graph then $G$ has a dominating pair of vertices.
\end{theorem}

\begin{proof} Suppose $G$ is a connected proper or unit interval $k$-graph and that the vertices have been labeled $v_1,v_2, \dots, v_n$ as in
Theorem \ref{DBO}.  We claim $\{v_1,v_n\}$ is a dominating pair.
Since $G$ is connected there is a path between $v_1$ and $v_n$; suppose the path is
$P =(v_1=v_{i_0}, v_{i_1}, v_{i_2}, \dots, v_{i_p} = v_n)$.  For each $k \in \{0, 1, 2, \dots, p-1\}$ and any $j$ satisfying $i_k \leq j \leq i_{k+1}$,
since $v_{i_{k}}v_{i_{k+1}} \in E(G)$, $v_j$ is adjacent to at least one of $v_{i_k}$ or $v_{i_{k+1}}$ because $v_{i_k},v_j,$ and $v_{i_{k+1}}$
belong to at east two different partite sets. We have proved every vertex of $G$ either belongs to $P$ or is adjacent to a vertex of $P$.
\end{proof}\\

\noindent The converse of Theorem \ref{dom_pair_in_propIkG} is not true since the vertices in the graph of Figure \ref{cockandballs}
labeled $2$ and $5$ are a dominating pair.
It is an open problem to determine what property characterizes those graphs which have a dominating pair of vertices.

Let $F$ be the graph of Figure \ref{cockandballs}.  We use $F$ and $F - 3$ (the graph $F$ with vertex $3$ deleted) 
to show that the class of class-proper interval $3$-graphs
is distinct from the classes of proper (or unit) interval $3$-graphs, and permutation graphs.  
See Figure \ref{fig:counterex_2to3}, first row, in which $F$, 
a class-proper representation for is is given as well as a permutation representation 
(the interval for $4$ contains that for $5$, but they are from different classes).  
The graph $F -3$
shows there are interval $3$-graphs whose complements are not circular arc graphs, Figure \ref{fig:counterex_2to3}, third row.
The second row of Figure \ref{fig:counterex_2to3} shows that $F$ is a cocomparability graph since
$\overline{F}$ has been given a transitive orientation shown via $\overline{F}_{tr}$.  

\begin{figure}[ht]
\begin{center}
\includegraphics[scale=0.4]{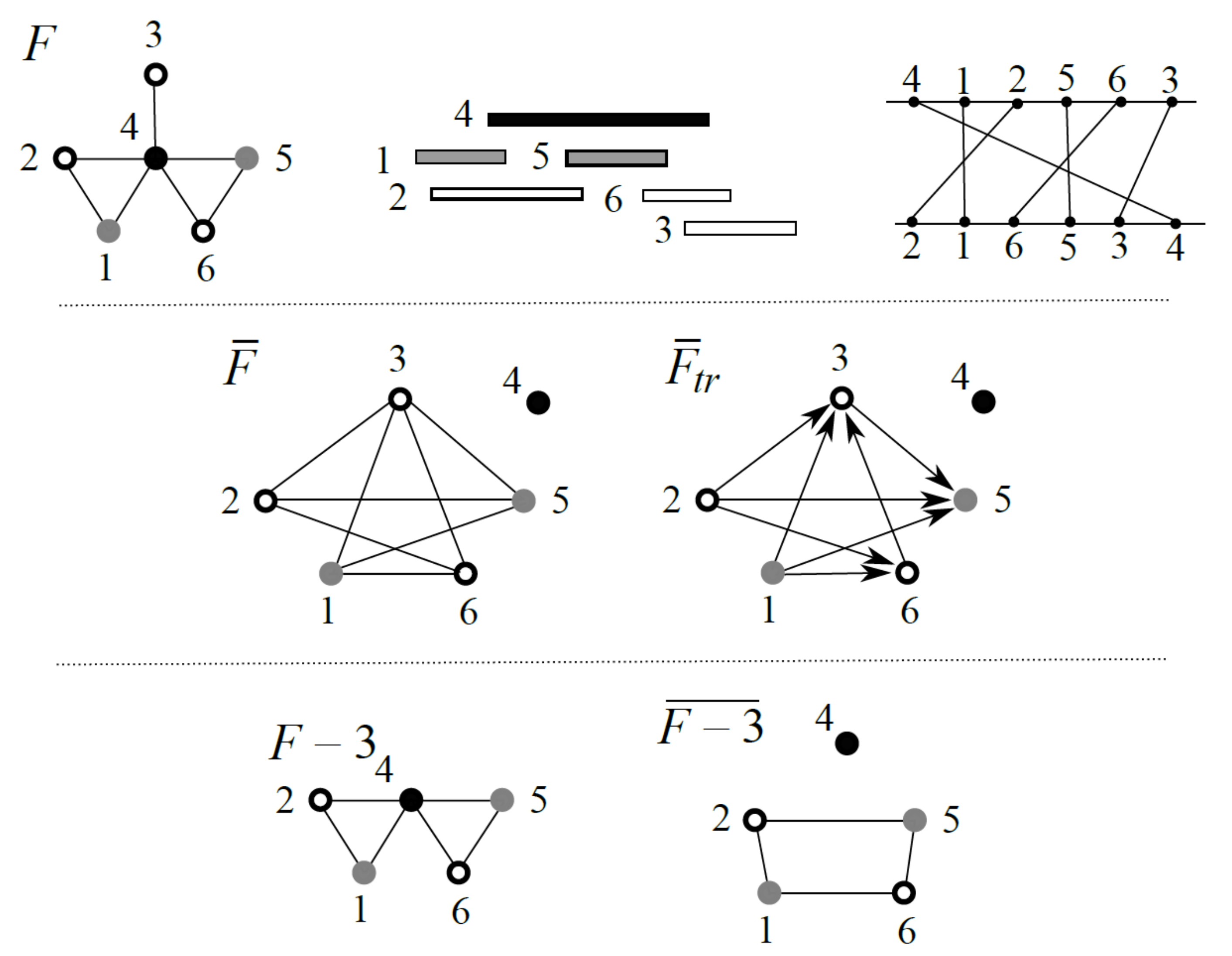}
\caption{Unit (or proper) interval $3$-graphs are not equivalent to class-proper interval $3$-graphs, permutation graphs,
are not equivalent to cocomparability graphs of posets of dimension three, and are not the complements of
(proper) circular arc graphs.}\label{fig:counterex_2to3}
\end{center}
\end{figure}

The vertices in the complement of a cocomparability interval $k$-graph can be covered with $k$ cliques and hence
any poset corresponding to a transitive orientation can be partitioned into $k$ chains.  So the width of the poset
is at most $k$ and by a theorem of Hiraguchi, the dimension of the corresponding poset is less than or equal to $k$.
But the converse is not true; that is, a poset of dimension less than or equal to $k$ does not necessarily have
an interval $k$-graph as an incomparability graph.  For example the graph in Figure \ref{fig:func_not_IkG}, $\overline{C}_6$,
is the incomparability graph of the $3$-crown in Figure \ref{fig:3-crown}, which has dimension three.
Furthermore, since $\overline{C}_6$ has no asteroidal triple, the analog to statement 11 of Theorem \ref{theorem:charPIBG} 
does hold for interval $k$-graphs, $k > 2$.  

We have achieved the goal of this section: to show that essentially no statement of Theorem \ref{theorem:charPIBG} can be extended 
to unit or proper interval $k$-graphs, for $k > 2$. 

\begin{figure}
\begin{center}
\includegraphics[scale=0.25]{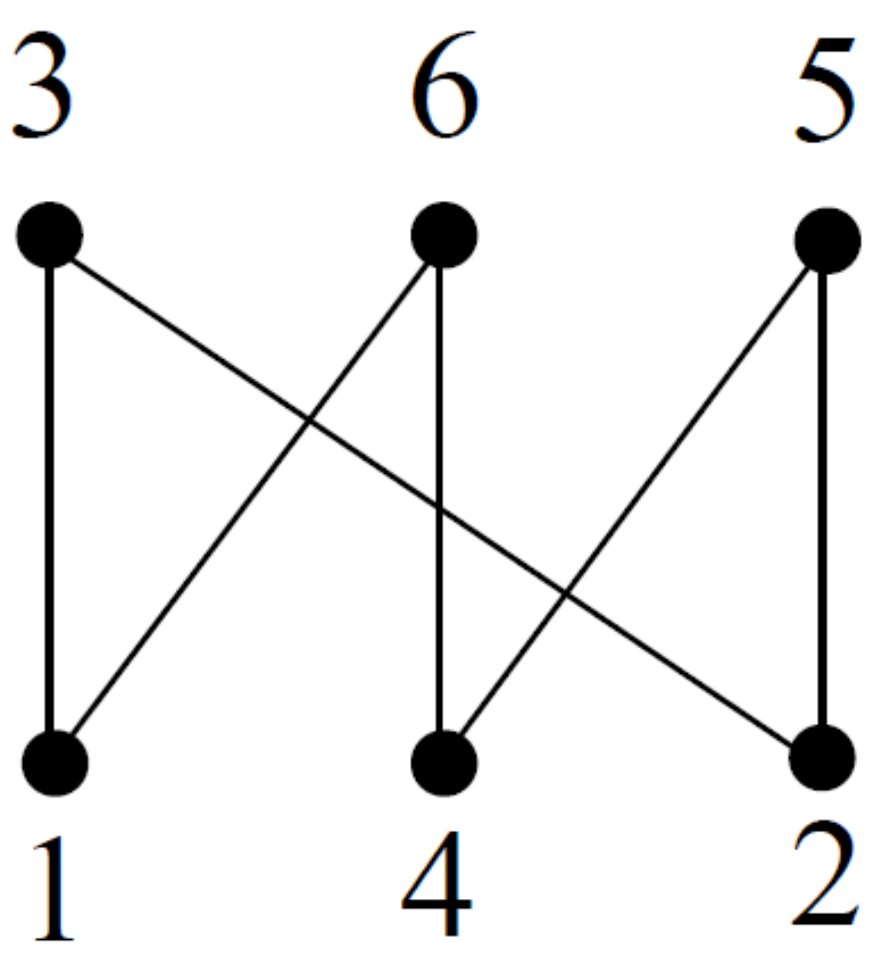}
\caption{The Hasse diagram of a poset of dimension $3$ that doesn't correspond to any cocomparability interval $k$-graph.}\label{fig:3-crown}
\end{center}
\end{figure}

\section{Interval $k$-Graphs and Orders, $k \geq 3$}

In this section we focus attention on $k \geq 3$ and will describe the structure of
orders corresponding to class-proper interval $k$-graphs.  We give two proofs that class-proper interval $k$-graphs
are cocomparability graphs, one using Theorem \ref{func=coco}.

\begin{theorem}\label{transCPIkG} If $G = (V,E)$ is a class-proper interval $k$-graph, for $k \geq 3$,
then $G$ is a cocomparability graph. \end{theorem}

\begin{proof} Assume $G$ is a class-proper interval $k$-graph with representation such that all interval endpoints are
distinct. We now use $I_v = [l(v),r(v)]$ to denote the interval for vertex $v$.
for vertex $v$. Index $V(G)$ as $v_1, v_2, \dots, v_n$ so that $i < j$ if and only if $r(v_i) < r(v_j)$.  Now orient the
edge $v_av_b \in E(\overline{G})$ as $v_a \to v_b$ if $r(v_a) < l(v_b)$ or $r(v_a) < r(v_b)$ and $v_a, v_b$ belong to
the same interval class.

We claim that this orientation is transitive. First, suppose $v_a \to v_b$, $v_b \to v_c$ and $v_av_c \not\in E(\overline{G})$.
We have $r(v_a) < r(v_b)<r(v_c)$ and so $l(v_c) < r(v_a)$ in order for $I_{v_a} \cap I_{v_c} \neq \O$.  But this means
$I_{v_b} \cap I_{v_c} \neq \O$, and so $I_{v_b},I_{v_c}$ belong to the same interval class.  Hence $I_{v_a}, I_{v_b}$ belong to
different interval classes. Now, unless $l(v_b) < l(v_c)$, the representation is not proper and $v_av_b \in E(G)$, a contradiction.

Now suppose $v_a \to v_b \to v_c \to v_a$ is assigned to $\overline{G}$.  Then $r(v_a) < r(v_b) < r(v_c) < r(v_a)$,
clearly a contradiction.  Therefore the orientation is transitive.

(\emph{Alternatively:}) We prove this via constructing a function representation for $G$; since function graphs
are cocomparability graphs, the result follows.

Begin with class-proper interval $k$-graph $G=(V,E)$ and find its interval representation
$\mathcal{I} = \{\mathcal{I}_1 \cup \mathcal{I}_2 \cup \cdots \cup \mathcal{I}_k\}$ in which each interval endpoint is distinct.
Now take $k+1$ horizontal lines
$L_0, L_1, \dots, L_k$ placed some distance apart from one another and place $\mathcal{I}$
on each one.  Define a function line $f_v$ corresponding to $v \in V$ with $I_v \in \mathcal{I}_i$.
The construction differs according to $i \in \{1, \dots, k-2\}$, $i = k-1$, and $i=k$.  Each function line is
defined by $k$ line segments, one with negative slope, one with positive slope, and the rest vertical.
For $i \in \{1,2,\dots, k-2\}$, connect via line segments $l(v)$ on $L_i$ to $r(v)$ on $L_{i+1}$ and $r(v)$ on $L_{i+1}$
to $l(v)$ on $L_{i+2}$, and use vertical line segments connecting $l(v)$s between the horizontal lines where $f_v$ has not been
defined.  For $i = k-1$, construct $f_v$ with a line segment connecting $l(v)$ on $L_{k-1}$ to $r(v)$ on $L_k$ and a line segment
connecting $r(v)$ on $L_0$ to $l(v)$ on $L_1$, then using vertical line segments to complete $f_v$.  For $i=k$, connect $l(v)$
on $L_0$ to $r(v)$ on $L_1$ and $r(v)$ on $L_1$ to $l(v)$ on $L_2$ and then vertical line segments for the rest of $f_v$.
See Figure \ref{fig:IkGfunction} for a depiction of this construction with $k=3$.

It is easy to verify that, for $u, v \in V$, if $uv \not\in E$, then $f_u$ and $f_v$ do not intersect
because $I_u,I_v \in \mathcal{I}_i$, or $I_u$ and $I_v$ are in different classes and do not intersect.
If $uv \in E$, then $f_u$ and $f_v$ intersect twice.  Now, orient $\overline{G}$ via $u \to v$ if and only if
$f_u(y) < f_v(y)$ for all $y$ in the domain (on the vertical axis) of the functions. Clearly, this is a transitive
orientation of $\overline{G}$; in fact it gives the same orientation as the one obtained above. \end{proof}\\

\begin{figure}
\begin{center}
\psfrag{I1}{{\small$\mathcal{I}_1$}}
\psfrag{I2}{{\small$\mathcal{I}_2$}}
\psfrag{I3}{{\small$\mathcal{I}_3$}}
\psfrag{L0}{$L_0$}
\psfrag{L1}{$L_1$}
\psfrag{L2}{$L_2$}
\psfrag{L3}{$L_3$}
\includegraphics[scale=0.6]{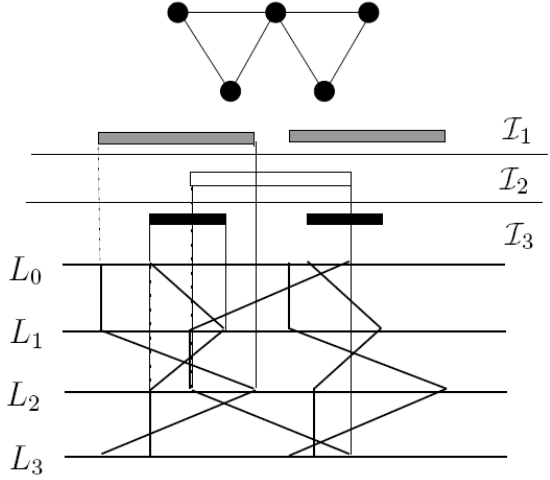}
\caption{A function representation for a class-proper interval $k$-graph.}\label{fig:IkGfunction}
\end{center}
\end{figure}

We now describe a vertex elimination scheme for class-proper interval $k$-graphs which in a sense generalizes the notion of consecutive
orderability of maximal cliques in interval graphs (cf. \cite{FulGro}).
Let $G$ be a class-proper interval $k$-graph with class-proper representation $\mathcal{I}$ in which all interval endpoints are distinct.
Order the vertices of $G$ as $v_1, v_2, \dots, v_n$ so that $r(v_i) < r(v_j)$ if and only if $i < j$. Now, observe that in $G$ all intervals
containing $r(v_1)$, including $I_{v_1}$, induce a complete multipartite subgraph in $G$.  Deleting the interval $I_{v_1}$ produces a
class-proper representation for $G - v_1$, and now the intervals containing $r(v_2)$ induce a complete multipartite subgraph in
$G-v_1$.  Clearly, this process may
be repeated so that $N_{H}[v_i]$ is complete multipartite in $H = G \setminus \{v_1, v_2, \dots, v_{i-1}\}$.  We record this observation below.

\begin{proposition}\label{IkGgraphelim} Let $G$ be an $n$-vertex class-proper interval $k$-graph with class-proper representation
$\{I_{v_i}: 1 \leq i \leq n\}$ where $V(G) = \{v_1, v_2, \dots, v_n\}$ with $r(v_i) < r(v_j)$ if and only if $i < j$.
Then, for $1\leq i \leq n$, $N_H[v_i]$ is complete multipartite in $H = G \setminus \{v_1, v_2, \dots, v_{i-1}\}$.
\end{proposition}

With $P = (V, \prec)$ the strict partial order obtained from a class-proper interval $k$-graph
$G = (V,E)$ and vertices ordered as in the
proposition above, we have $v_i \prec v_j \implies i < j$.  Also, translating the above proposition into ordered set parlance,
we have Corollary \ref{IkGorder_elim}.  Denote by $\overline{N}_P(x)$ the set of elements incomparable with $x$ in $P$.
When we say an order (or suborder) has a \emph{decomposition into chains}, we mean that the elements of the order (or suborder)
may be partitioned into chains $C_1, C_2,  \dots, C_m$, with $x \in C_i$ incomparable to $y \in C_j$ whenever $i \neq j$.

\begin{corollary}\label{IkGorder_elim} Let $P=(V,\prec)$ be a strict partial order whose incomparability graph is a class-proper interval
$k$-graph $G = (V,E)$ and $V$ is indexed as in Proposition \ref{IkGgraphelim}.
Then, with $P' = P \setminus \{v_1, v_2, \dots, v_{i-1}\}$, $v_i$ is minimal in $P'$ and
$\overline{N}_{P'}(v_i)$ can be decomposed into chains.
\end{corollary}

\noindent\emph{Remark.} In Figure \ref{fig:cocoIkG_notPIG} an ordering of $V(M)$ which corresponds to the
prescriptions of Proposition \ref{IkGgraphelim} is $v_1 = a, v_2 = x, v_3 = p, v_4 = y, v_5 = b, v_6 = q$.
Note that this order may be obtained also from $P$ via Corollary \ref{IkGorder_elim}.
Begin by finding a minimal element whose set of incomparable elements can be partitioned into chains with no comparabilities between chains.
For example $a$ has $\overline{N}_{P}(a) = \{x,y,p\}$ with $C_1$ being the 2-chain $x \prec y$, and $C_2$ the 1-chain $p$.
So $a$ is a suitable first element.
The next element in the ordering must be $x$, since $\overline{N}_{P-\{a\}}(p) = \{x,y,b\}$ and $x \prec y$, $x \prec b$, and $y \| b$.\\

We now proceed from the other perspective to the end of characterizing interval $k$-orders.
However, and in distinction to the $k=2$ case, the assignment of elements of the order to classes
(vertices of the incomparability graph to color/interval classes) must be done with care, since the class
assignment (coloring of the incomparability graph) is not forced in the $k > 2$ circumstance.
Figure \ref{fig:M_troublesome} is intended to illustrate this problem.
In spite of this we will prove the following theorem after we prove that an appropriate color/interval assignment can be found.

\begin{theorem}\label{IkOimpliesCPIkG} Let $P = (V, \prec)$ be a strict order with $V$ labeled $v_1, v_2, \dots, v_n$ so that $v_i$ is minimal in
$P'=P \setminus \{v_1, v_2, \dots, v_{i-1}\}$ and $\overline{N}_{P'}(v_i)$ can be decomposed into chains.  Then the incomparability
graph of $P$ is a class-proper interval $k$-graph. \end{theorem}

To prove this theorem, we give a class-proper interval representation derived from an order satisfying
the hypothesis of Theorem \ref{IkOimpliesCPIkG} as follows.
Define $\mu(v_i) = \min\{j: v_i \| v_j\}$, and put $I_{v_i} = \left[\mu(i) - \left(1-\frac{i}{n}\right),i\right]$.
Note that $i \in \{j: v_i \| v_j\}$ since $\prec$ is strict, and so $I_{v_i}$ is well-defined.
We may assume, appealing to Dilworth's theorem \cite{Dil} and the fact that cocomparability graphs are perfect
(which follows essentially from Dilworth's theorem and the fact that the class of perfect graphs is closed under complementation, see \cite{Lov}),
that $\mathrm{width}(P) = \chi(G) =k$, and that the independent sets of $G$ correspond to the $k$ chains that cover $P$.
Now, referring to Figure \ref{fig:M_troublesome}, the order $\mathcal{M}$ can be partitioned into chains in three different ways.
One of these partitions together with the above interval representation construction will not give back the order desired.
In particular $\mathcal{P}_3$, the third covering of $\mathcal{M}$ in the figure,
yields $v_3$ incomparable to the rest of the order when we should have $v_3 \prec v_5$.
Consequently we must choose the covering carefully, whence the following claim.\\

\begin{figure}[ht!]
\begin{center}
\includegraphics[scale=0.4]{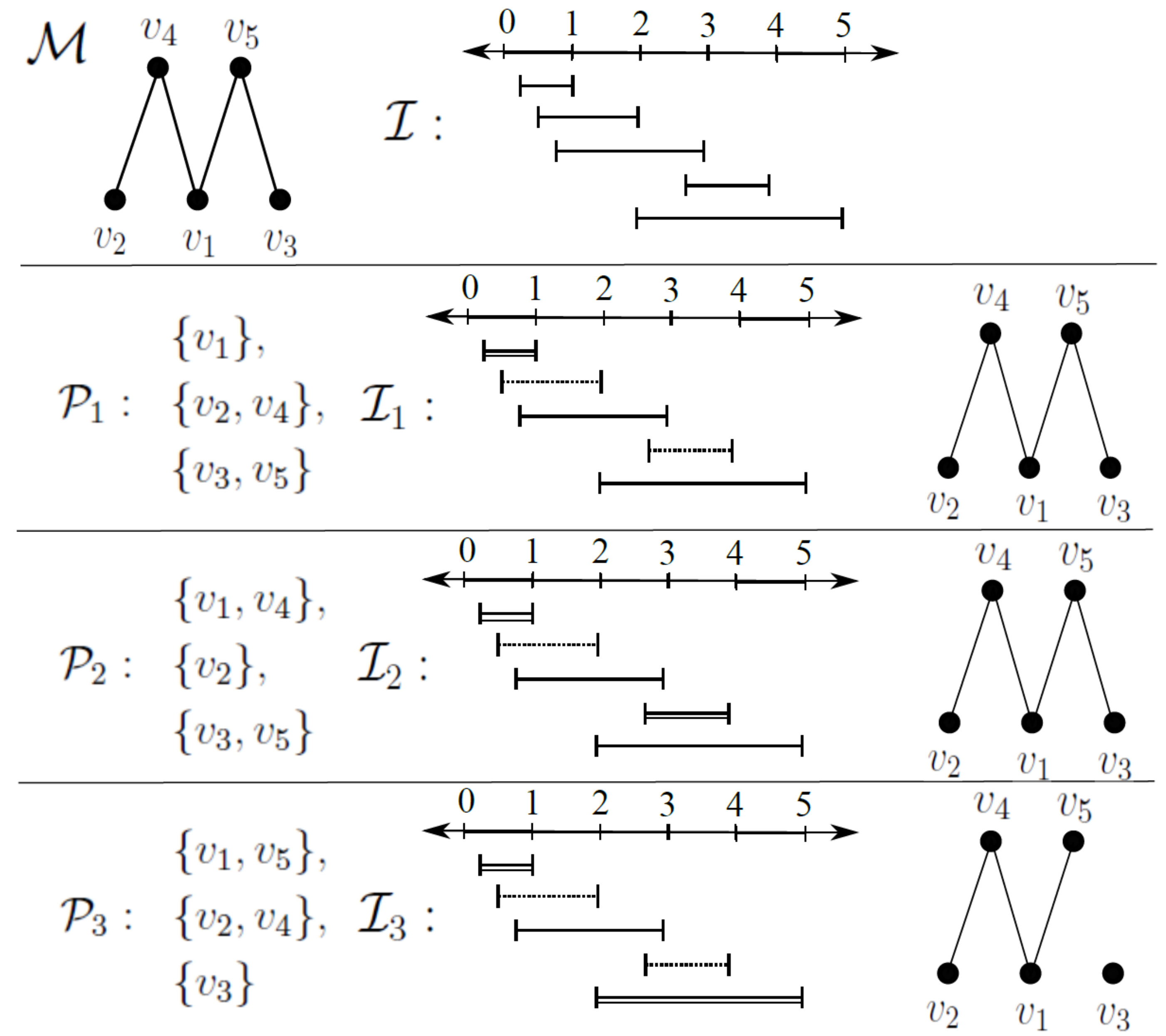}
\caption{The order $\mathcal{M}$ and corresponding intervals constructed in accord with Theorem \ref{IkOimpliesCPIkG}.  Also
the different interval $k$-orders as a function of the three ways $\mathcal{M}$ can be partitioned into chains. (We have
solid, double-solid, and dashed intervals to represent the class assignment given by the indicated partition into chains.)}\label{fig:M_troublesome}
\end{center}
\end{figure}

\noindent {\bf Claim:} \emph{There exists a covering of $P$ by chains such that the chains in the decomposition of $\overline{N}_{P'}(v_i)$
are each contained in a unique class.
That is, no chain of $\overline{N}_{P'}(v_i)$ contains vertices of two or more classes.}\\

\noindent{\bf Proof of Claim.}  Suppose we have any covering.
Let $i$ be the smallest subscript where $v_i$ fails the conditions of the claim.
We will change the covering so $v_1, \dots, v_i$ meet the conditions of the claim,
by successively increasing $i$.

Let $x \prec y$ be two vertices in a chain in the decomposition of $\overline{N}_{P'}\left(v_i\right)$ where $x$ and $y$ are in different classes,
there are no vertices between $x$ and $y$, but all vertices that precede $x$ in the chain are in the same class as $x$.
Let the class that contains $x$ be $X = x_1 \prec x_2 \prec \dots \prec x_j = x \prec \dots \prec x_m$,
and the class that contains $y$ be $Y = y_1 \prec y_2 \prec \dots \prec y_l = y \prec \dots \prec y_n$.

Observe that  $x_{j+1}$ (if it exists) must satisfy $v_i \prec x_{j+1}$, or $y \prec x_{j+1}$;
otherwise $\overline{N}_{P'}\left(v_i\right)$ would contain a component that is not a chain.

Observe also that, if $y_{l-1}$ exists, either $y_{l-1} \prec v_i$, $y_{l-1} \prec x_{j}$ or, if neither of those are true,
$y_{l-1}$ is contained among $v_1, v_2 \dots v_{i-1}$.
For each of these cases, if both $y_{l-1}$ and $x_{j+1}$ exist, $y_{l-1} \prec x_{j+1}$ by transitivity and the observation above.

We form a new covering by replacing $X$ and $Y$ with two new classes as follows:
$X' = x_1 \prec \dots \prec x \prec y \prec \dots \prec y_n$, and
$Y' = y_1 \prec \dots \prec y_{l-1} \prec x_{j+1} \prec \dots x_m$.

We need to check that the new covering has several properties.
First observe that, because of the structure of $\overline{N}_{P'}\left(v_i\right)$,
no other chain in the decomposition of $\overline{N}_{P'}\left(v_i\right)$ has vertices in $X$ or $Y$,
so their classes remain unaltered at this step.
Second observe that, in the chain in $\overline{N}_{P'}\left(v_i\right)$ that contains $x$ and $y$,
all of the vertices that precede $x$ must be in $X$ and consequently become part of $X'$
(and thus we can repeat this process until the entire chain is in a single class, if necessary).
Finally we need to make certain that, for any $v_{r}$ with $r < i$, all chains in $\overline{N}_{P'}\left(v_{r}\right)$ stay in unique classes.

Suppose a chain in $\overline{N}_{P'}\left(v_r\right)$ does break into two classes.
Then this chain must contain either $x = x_{j} \prec x_{j+1}$ or $y_{l-1} \prec y_{l}$.

\smallskip

\underline{Case 1}:  The chain contains $x_{j} \prec x_{j+1}$.
Recall that either $y_l \prec x_{j+1}$ or $v_i \prec x_{j+1}$.
In the first case since $x_{j} \prec y \prec x_{j+1}$, there would be a chain in $\overline{N}_{P'}\left(v_r\right)$
split between $X$ and $Y$ which contradicts our choice of $i$.
In the second case, since $v_r || x_{j+1}$ it follows that  $v_r || v_i$ (since $r < i$, we know that $v_i \prec v_r$ is impossible).
However the vertices $x_{j}, x_{j+1}, v_i$ are all in $\overline{N}_{P'}\left(v_r\right)$ but do not form a chain, contradicting our labeling.

\smallskip

\underline{Case 2}:  The chain contains $y_{l-1} \prec y_{l}$.
If $y_{l-1} \prec x_j$, then $y_{l-1} < x_j < y_{l}$ for part of the chain in $\overline{N}_{P'}\left(v_r\right)$
that is divided between $X$ and $Y$, contradicting our choice of $i$.
Suppose then that $y_{l-1}$ and $x$ are incomparable.
By the ordering of the vertices $x \not \prec v_r$, however, since $y_{l} || v_r$, $v_r \not \prec x$ either.
Consequently the vertices $y_{l-1}$, $y_l$, and $x_j$ are part of a component of $\overline{N}_{P'}\left(v_r\right)$ which is not a chain,
contradicting our assumptions.

\smallskip

By successively forming these new coverings we make each chain of $\overline{N}_{P'}\left(v_i\right)$ fit in a unique class
and ultimately find a covering that satisfies the conditions of the claim.  This proves the claim. \hfill $\Box$\\

\noindent We return to the proof of Theorem \ref{IkOimpliesCPIkG} and verify that the interval representation has the requisite properties.

\medskip

{\bf Claim:} The intervals are class proper.

\smallskip

\noindent\emph{Proof of Claim.} Note that no interval is empty or trivial.
Now, suppose some interval contains another from its class.
That is, suppose there are elements $v_a, v_d$ from the same chain of $P$ with $v_a \prec v_d$ and $b = \mu(v_d), \mu(v_a) = c$, and $b < c$.
Hence, $l(v_d)< l(v_a)$.
But $v_d$ must be comparable to all elements $v_i$ with $i < b$, and $v_a$ must be comparable with all elements $v_i$, with $i< c$.
So $v_b \prec v_a$ and transitivity forces $v_b \prec v_d$, contradicting $\mu(v_d) = c$.

\medskip

\noindent{\bf Claim:} If $v_i \| v_j$, then $I_{v_i} \cap I_{v_j} \neq \O$.

\smallskip

\noindent\emph{Proof of Claim.} Suppose $v_i$ and $v_j$ are distinct vertices with $v_i \| v_j$ and $i <j$.
Then $\mu(i),\mu(j) \leq i$ and both intervals contain the segment $[i - (1-j/n),i]$.

\medskip

\noindent{\bf Claim:} If $v_i$ and $v_j$ are comparable, then $v_iv_j \not\in E(G)$.

\smallskip

\noindent\emph{Proof of Claim.}  If $v_i, v_j$ belong to the same chain in $P$, then their intervals belong to the same
class and do not induce adjacency regardless of whether they intersect.
So suppose $v_i \in C_r$, $v_j \in C_s$, where $r \neq s$, and say $v_i \prec v_j$, hence $i < j$.




\medskip

\noindent{\bf Claim:} If $x \prec y$ or $y \prec x$, then $I_x \cap I_y = \O$, unless $I_x,I_y \in \mathcal{I}_j$.

\smallskip

\noindent \emph{Proof of Claim.} Suppose $x \prec y$ and $x$ and $y$ are in different classes.
Let $x = v_i$ and $y = v_j$ (and necessarily $i < j$).
We need to check that $i < \mu(j)$.
Let $k = \mu(j)$, and suppose $i \geq k$.
Since $v_k \| y$, $x \neq v_k$, so $k < i < j$.
In this case, $x = v_i \not \prec v_k$, by the structure of the labeling.
On the other hand if $v_k \prec x$, by transitivity $v_k \prec y$, a contradiction.
Consequently $v_k \| x$ as well.
However, now $x$ and $y$ are in the same chain of $\overline{N}_{P'}\left(v_k\right)$, and therefore must be in the same class.

\noindent The proof of the theorem is complete. \hfill \rule{2mm}{2mm}\\

\subsection{Characterization of Interval $3$-Orders by One Obstruction}

We have identified the mechanism by which a transitive orientation of the complement of an
interval $k$-graph is prohibited: that an interval contain another from its class properly.
In this section we characterize interval $3$-graphs that are cocomparabilty graphs via one forbidden induced subgraph
(the complement of a $6$-cycle) and
consequently also the interval $3$-orders by one forbidden suborder (the order often referred to as the $3$-crown).

\begin{theorem} A $3$-chromatic cocomparability graph is a class-proper interval $3$-graph
if and only if it contains no subgraph isomorphic to $\overline{C}_6$ of Figure \ref{fig:func_not_IkG}.
\end{theorem}

\begin{proof} Let $G$ be a $3$-chromatic cocomparability graph and note that no vertex of
$\overline{C}_6$ has an induced complete multipartite neighborhood, so $G$ cannot have $\overline{C}_6$
as an induced subgraph.  Also note that any cocomparability graph on fewer than six vertices is an interval
$k$-graph, so suppose $|V(G)| \geq 6$.

Now assume $G$ is minimal counterexample to the result in that the neighborhood
of no vertex of $G$ induces a complete multipartite neighborhood.  Let $P=(V(G), \prec)$ be the
poset obtained from Corollary 3.1.  Since $G$ is $3$-chromatic,
and by Dilworth's theorem, $P$ can be decomposed into three maximal chains, say $C_1, C_2$, and $C_3$.
Now we argue by the number of minimal elements of $P$.

If $P$ has one minimal element, say $x$, then $x$ is isolated in $G$ and its neighborhood is complete multipartite.
If $P$ has two minimal elements, say $x$ and $y$ with (relabeling if necessary) $x \in C_1$ and $y \in C_2$.
Let $z$ be the minimal element of $C_3$, but $z$ is not minimal in $P$, so without loss of generality $x<z$.
Then $x$ is incomparable with a subchain of $C_2$ and is complete multipartite. [Details: Suppose $x$ is minimal in $C_1$ and
$y$ is minimal in $C_2$. By design $x \| y$. Let $C_2$ consist of $y \prec y_1 \prec \cdots \prec y_k$.  Then we may suppose
$y_k \succ x$, and so $x \| y,y_1,y_2,\dots, y_{k-1}$.  $N[x] \cong K_{1,k-1}$.]
By dual arguments we can also determine that $P$ must have $3$ maximal elements.


Suppose $P$ has exactly three minimal elements, $x \in  C_1, y \in C_2, z \in C_3$.
$P$ must have exactly $3$ maximal elements as well.
None of $x, y, z$ have induced complete multipartite neighborhoods in $G$, otherwise $G$ is not minimal as assumed.
We know that each of $C_1$, $C_2$, and $C_3$ have more than one element, otherwise $x$ or $y$ or $z$ is both a minimal and maximal element of $P$.
Suppose $x$ is such an element.  Then $x$ is isolated in $P$, and either $y$ or $z$ has a complete multipartite neighborhood in $G$.

Suppose $C_1$ consists of $x \prec x_1 \prec  \dots \prec x_r$, $C_2$ consists
of $y \prec y_1 \prec \dots \prec y_s$, and $C_3$ consists of $z \prec z_1 \prec \dots \prec z_t$.
Relabeling if necessary, we may assume there is an element $y_i$, $i \geq 1$ with $x \| y_i$ and $z \prec y_i$.
Now $N[z]$ is not complete multipartite.
If there is a $y_j$, $1 \leq j \leq i$, with $y_j \| z$ and $x \prec y_j$, then $y_j \prec y_i$ and hence $x \prec y_i$, a contradiction.
So there is an $x_k$, $k \geq 1$, with $x_k \| z$ and $y \prec x_k$.
Now, $N[y]$ is not complete multipartite in $G$.
If there is an $x_l$ incomparable with $y$ and with $z \prec x_l$, then $x_l \prec x_k$, a contradiction.
Thus there is an element $z_m$ incomparable with $y$ and with $x \prec z_m$.  But now the elements $x,z_m, z, y_i, y, x_k$ are related
so that $P$ contains an induced $6$-cycle; that is, $G$ contains and induced $\overline{C}_6$.\end{proof}\\

We end with a conjecture.\\

\noindent{\bf Conjecture:} \emph{If $G$ is a cocomparability graph, then
$G$ is an interval $k$-graph if and only if it has no induced subgraph isomorphic to $\overline{C}_6$
or $\overline{2P}_3$ (cf. Figure \ref{fig:Flink}).}

\begin{figure}
\begin{center}
\includegraphics[scale=0.25]{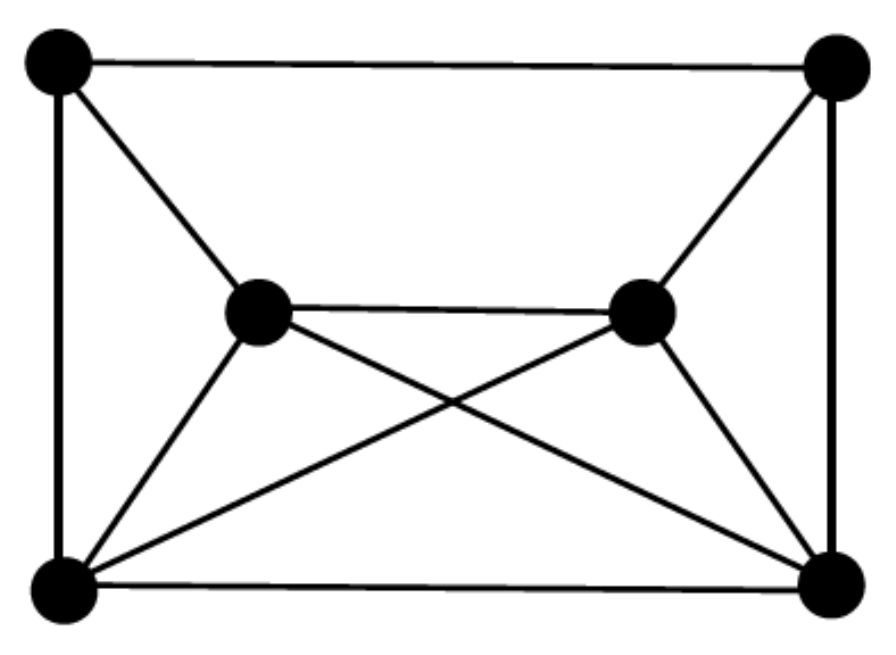}
\caption{The graph $\overline{2P}_3$; we conjecture it and $\overline{C}_6$ are the forbidden subgraphs
that characterize the incomparability graphs of interval $k$-orders.}\label{fig:Flink}
\end{center}
\end{figure}

\bibliographystyle{amsplain}
\bibliography{IkGOrders_rev}

\providecommand{\bysame}{\leavevmode\hbox to3em{\hrulefill}\thinspace}
\providecommand{\MR}{\relax\ifhmode\unskip\space\fi MR }
\providecommand{\MRhref}[2]{%
  \href{http://www.ams.org/mathscinet-getitem?mr=#1}{#2}
}
\providecommand{\href}[2]{#2}
\begin{thebibliography}{10}

\bibitem{BraSpiSte}
A.~Brandst{\"{a}}dt, J.~Spinrad, and L.~Stewart, \emph{Bipartite permutation
  graphs}, Discrete Applied Math. \textbf{18} (1987), 279--292.

\bibitem{BraSpiSte2}
\bysame, \emph{Bipartite permutation graphs are bipartite tolerance graphs},
  Congressus Numerantium \textbf{58} (1987), 165--174.

\bibitem{BroFle}
D.~E. Brown and B.~M. Flesch, \emph{A characterization of 2-tree proper
  interval 3-graphs}, Journal of Discrete Mathematics \textbf{Article ID
  143809} (2014).

\bibitem{BroLanPIO}
D.~E. Brown and L.~J. Langley, \emph{The mathematics of preference, choice and
  order: Essays in honor of peter c. fishburn}, ch.~Probe Interval Orders,
  pp.~313--322, Springer-Verlag Heidelberg Berlin, 2009.

\bibitem{BroLunUIBG}
D.~E. Brown and J.~R. Lundgren, \emph{Characterizations for unit interval
  bigraphs}, Congressus Numerantium \textbf{206} (2010), 5 -- 17.

\bibitem{Bro}
D.E. Brown, \emph{Variations on interval graphs}, Ph.D. thesis, University of
  Colorado Denver, 2004.

\bibitem{BroLunAus}
D.E. Brown and J.R. Lundgren, \emph{Bipartite probe interval graphs, interval
  point bigraphs, and circular arc graphs}, Australasian J. Combinatorics
  \textbf{35} (2006), 221--236.

\bibitem{BroLunShe}
D.E. Brown, J.R. Lundgren, and L.~Sheng, \emph{Cycle-free unit and proper probe
  interval graphs}, submitted to Discrete Applied Math.

\bibitem{CorOlaSte}
D.G. Corneil, S.~Olariu, and L.~Stewart, \emph{Asteroidal triple-free graphs},
  SIAM J. Discrete Math. \textbf{10} (1997), 399--430.

\bibitem{DasRoySenWes}
S.~Das, A.B. Roy, M.~Sen, and D.B. West, \emph{Interval digraphs: an analogue
  of interval graphs}, Journal of Graph Theory \textbf{13} (1989), no.~2,
  189--202.

\bibitem{Dil}
R.~P. Dilworth, \emph{A decomposition theorem for partially ordered sets},
  Annals of Mathematics \textbf{51} (1950), 161--166.

\bibitem{DusMil}
Ben Dushnik and E.W. Miller, \emph{Partially ordered sets}, Amer. J. Math.
  \textbf{63} (1941), 600--610. \MR{MR0004862 (3,73a)}

\bibitem{Fis}
P.~C. Fishburn, \emph{Interval orders and interval graphs}, Wiley \& Sons,
  1985.

\bibitem{FulGro}
D.R. Fulkerson and O.A. Gross, \emph{Incidence matrices and interval graphs},
  Pacific J. Math. \textbf{15} (1965), 835--855.

\bibitem{Gal}
T.~Gallai, \emph{Transitiv orientbare graphen}, Acta Math Acad. Sci. Hungar
  \textbf{18} (1967), 25--66.

\bibitem{GolRotUrr}
M.C. Golumbic, D.~Rotem, and J.~Urrutia, \emph{Comparability graphs and
  intersection graphs}, Discrete Math \textbf{43} (1983), 37--46.

\bibitem{HelHua}
P.~Hell and J.~Huang, \emph{Interval bigraphs and circular arc graphs}, Journal
  of Graph Theory \textbf{46} (2004), 313--327.

\bibitem{Lov}
L.~Lov\'{a}sz, \emph{A characterization of perfect graphs}, Journal of
  Combinatorial Theory, Series B \textbf{13} (1972), 95 --98.

\bibitem{McKMcM}
T.~McKee and F.R. McMorris, \emph{Topics in intersection graph theory}, Society
  for Industrial and Applied Mathematics, Philadelphia, 1999.

\bibitem{McMWanZha}
F.R. McMorris, C.~Wang, and P.~Zhang, \emph{On probe interval graphs}, Discrete
  Applied Mathematics \textbf{88} (1998), 315--324.

\bibitem{SanSen}
B.K. Sanyal and M.K. Sen, \emph{Indifference digraphs: a generalization of
  indifference graphs and semiorders}, SIAM J. Discrete Math. \textbf{7}
  (1994), no.~2, 157--165.

\bibitem{She}
L.~Sheng, \emph{Cycle-free probe interval graphs}, Congressus Numerantium
  \textbf{88} (1999), 33--42.

\bibitem{Spi}
J.~Spinrad, \emph{Circular-arc graphs with clique cover number two}, J. Comb.
  Theory, Series B \textbf{44} (1987), no.~3, 300--306.

\end{thebibliography}
\end{document}